\documentclass[11pt]{amsart}
\usepackage[utf8]{inputenc}
\usepackage{color}
\usepackage{tikz}
\usetikzlibrary{patterns}
\usepackage[margin=2.5cm]{geometry}
\usepackage{verbatim}
\usepackage[colorinlistoftodos,prependcaption,textsize=tiny]{todonotes}

\title[Inside factorial monoids]{Inside factorial monoids and the cale monoid of a single Diophantine equation}
\author{Pedro A. García-Sánchez}
\address{Departamento de Álgebra and IEMath-GR, Universidad de Granada, E-18071 Granada, España}
\email{pedro@ugr.es}
\author{Ulrich Krause}
\address{Universität Bremen, Fachberich Mathematik/Informatik, D-28359 Bremen, Germany}
\email{ukrause@uni-bremen.de}
\author{David Llena}
\address{Departamento de Matem\'{a}ticas, Universidad de Almeria, E-04120 Almeria,  Espa\~na}
\email{dllena@ual.es}

\thanks{The first and third authors are supported by the project MTM2017-84890-P, which is funded by Ministerio de Econom\'{\i}a, Industria y Competitividad and Fondo Europeo de Desarrollo Regional FEDER, and by the Junta de Andaluc\'{\i}a Grant Number FQM-343. Part of this work was done while the second author visited the Universities of Almería and Granada supported by the project MTM2014-55367-P}
\date{\today}



\newtheorem{proposition}{Proposition}
\newtheorem{lemma}{Lemma}
\newtheorem{theorem}{Theorem}
\newtheorem{corollary}{Corollary}

\theoremstyle{remark}
\newtheorem{example}{Example}
\newtheorem{remark}{Remark}

\newcommand{\NN}{\mathbb N}
\DeclareMathOperator{\Ap}{Ap}
\def\peb[#1]{{\left\lfloor #1\right\rfloor}}
\def\pe[#1]{{\left\lceil #1\right\rceil}}

\begin{document}

\maketitle

\begin{abstract}
    We give a structure theorem for inside factorial domains. As an example we study the monoid of nonnegative integer solutions of equations of the form $a_1x_1+\cdots +a_{r-1}x_{r-1}=a_rx_r$, with $a_1,\ldots,a_r$ positive integers. This set is isomorphic to a simplicial full affine semigroup, and thus it can be described in terms of its extremal rays and the Ap\'ery sets with respect to the extremal rays. 
\end{abstract}

\section{Introduction}

The motivation of this manuscript was essentially the extension of Elliott's results on the sets of nonnegative integer solutions of equations of the form $a x+ b y=c z$, with $a,b,c$ positive integers. Elliott was able to give parametric solutions to these equations for $c$ up to 10. The set of solutions of these equations are full affine semigroups, and consequently they can be ``parametrized'' in terms of the Ap\'ery sets of the extremal rays. Simplicial full affine semigroups are a particular instance of Cale monoids, which are always inside factorial. Also every full affine semigroup is isomorphic to a normal affine semigroup. The normal condition is known as root-closed in multiplicative notation.

We see that for root-closed inside factorial monoids, elements can be expressed uniquely as a linear combination of the elements in a basis plus an element in the Ap\'ery set of this basis. Thus basis for inside factorial monoids play the same role as extremal rays in simplicial full affine semigroups. This motivates the extension of the structure theorem for simplicial full affine semigroups to inside factorial monoids. It turns out that root-closed inside factorial monoids are of the form $G\times F$ with $G$ a torsion Abelian group and $F$ a free monoid, endowed with an operation that has a ``carry'' on the first component. 

Cale monoids are a particular case of inside factorial monoids with special bases known as tame bases. For these monoids we are able to determine the class group and the inner class group. We see that the inner class group coincides with the Ap\'ery set of the monoid with respect to a tame base, and actually the Ap\'ery set can be endowed with an operation that makes it isomorphic to the inner class group. 

Going back to the original problem of studying the monoid of nonnegative integer solutions of $a_1x_1+\dots+ a_{r-1}x_r=a_rx_r$, with all $a_i$ positive integers, we particularize the results obtained above and recover some known results for simplicial affine semigroups. For the particular case of $r=3$, we see that the inner class group is cyclic, and we are able to compute it together with the class group. We discuss with some examples the uniqueness achieved by Elliott in the description of the set of solutions for the case $r=3$.

\section{Inside factorial monoids}

Let $M$ be a cancellative and commutative monoid, and let $Q$ be a set of nonunit elements of $M$. We say that $M$ is \emph{inside factorial with base} $Q$ if 
\begin{itemize}
    \item for each element $x\in M$, there exists a positive integer $n(x)$ and $u$ a unit in $M$ such that $n(x)x\in u+\langle Q\rangle$,
    \item if $n(x)x =u+\sum_{q\in Q}\lambda_q q = v + \sum_{q\in Q} \mu_q q$, for some $u$, $v$ units of $M$ and $\lambda_q, \mu_q\in \mathbb N$ (only a finite number of them is nonzero; $\mathbb{N}$ stands for the set of nonnegative integers), then we have $\lambda_q=\mu_q$ for all $q\in Q$. 
\end{itemize}
The set $Q$ is known as a \emph{Cale basis} for $M$. 

If we assume that $M$ is unit free (reduced), then $M$ is an inside factorial monoid with base $Q$ if and only if 
for every $x\in M$, there exist nonnegative integers $n(x)$, $x(q)$, $q\in Q$, such that 
\[
n(x) x=\sum_{q\in Q} x(q) q,
\]
with all $x(q)$ equal to zero except for finitely many $q\in Q$, and the elements in $Q$ are $\mathbb{Q}$-linearly independent.

For a cancellative monoid $M$ we can define the divisibility relation as follows: 
\[a\le_M b \hbox{ if there exists } c\in M \hbox{ such that }a+c=b.\] 

Given a reduced inside factorial monoid $M$, we say that $q\in M$ is a \emph{strong atom} if the only nontrivial divisors of multiples of $q$ are multiples of $q$. That is, if for some $n\in\mathbb Z^+$ and $x \in M$ with $x\le_M nq$, then $x=mq$ for some $m\in \mathbb Z^+$. The set of all strong atoms of $M$ is a Cale base for $M$ (this is a consequence of \cite[Lemma 3.3 (4)]{cale}).

Let $M$ be an inside factorial monoid with base $Q$, we define the \emph{Ap\'ery set} of $M$ with respect to $Q$ as follows
\[
\Ap(M,Q)=M\setminus(Q+M).
\]

\begin{lemma}\label{lem:lambda-lt-1}
Let $M$ be a reduced inside factorial monoid with basis $Q$. Then
\[
\{x\in M \mid \lambda(q,x)<1 \hbox{ for all } q\in Q\}\subseteq \Ap(M,Q),
\]  
where for $x,y\in M$, $\lambda(x,y)=\sup\{\frac{m}n \mid m x\le_M n y\}$.
\end{lemma}
\begin{proof}
    Assume that $\lambda(q,x)<1$ for all $q\in Q$ and suppose that $x=q+y$ for some $y\in M$. Then $q\le_M x$, and consequently $\lambda(q,x)\ge 1$, a contradiction. Thus, $x\in \Ap(M,Q)$.
\end{proof}

Let $M$ be a cancellative monoid. Then $M$ can be embedded naturally in its quotient group $\mathrm G(M)$. We say that $M$ is \emph{root closed} if whenever $n z\in M$ for some positive integer $n$ and some $z\in \mathrm G(M)$, then $z\in M$.

We see next that if we add the root-closed condition to Lemma \ref{lem:lambda-lt-1}, then we get an equality. Clearly (as was already used in Lemma \ref{lem:lambda-lt-1}), if $q\le_M x$, then $\lambda(q,x)\ge 1$. The converse holds under the root-closed condition.

\begin{lemma}\label{lem:lambda-ge-1}
Let $M$ be a reduced root-closed inside factorial monoid with base $Q$. For $x\in M$ and $q\in Q$,  $\lambda(q,x)\ge 1$ if and only if $q\le_M x$. 
\end{lemma}
\begin{proof}
By hypothesis, there exists two positive integers $m$ and $n$ such that $m q \le_M n x$, with $m\ge n$. Then $n q \le_M m q\le_M n x$, which yields $n x-n q\in M$. Hence $n(x-q)\in M$ and as $M$ is root closed, $x-q\in M$, equivalently, $q\le_M x$.

Now assume that $q\le_M x$. Then $\lambda(q,x)\ge \frac{1}1=1$.
\end{proof}

\begin{lemma}\label{lem:ap-lambda}
Let $M$ be a reduced and root-closed inside factorial monoid with basis $Q$. Then
\[
\Ap(M,Q)=\{x\in M \mid \lambda(q,x)<1 \hbox{ for all } q\in Q\}.
\]  
\end{lemma}
\begin{proof}
    Let $x \in \Ap(M,Q)$. Then $x\not\in q+M$ for all $q\in Q$. Assume that $\lambda(q,x)\ge 1$ for some $q\in Q$. By Lemma \ref{lem:lambda-ge-1}, $x-q\in M$. But this means that $x\in q+M$, contradicting that $x\in \Ap(M,Q)$.
    
    The other inclusion was shown in Lemma \ref{lem:lambda-lt-1}.
\end{proof}

Given $x\in M$, we can define 
\[
\nu(x)=\sup\{\lambda(q,x)\mid q\in Q\}. 
\]
Notice that the above supremum in our case is a maximum, since for every $x$ only for finitely many $q$, $\lambda(q,x)$ will be nonzero (by \cite[Lemma 2 (a)]{inside}, $\lambda(q,x)=n(x)/x(q)$).

Lemma \ref{lem:ap-lambda} can be restated as follows in terms of $\nu$.

\begin{corollary}
Let $M$ be a reduced and root-closed inside factorial monoid with basis $Q$. Then
\[
\Ap(M,Q)=\nu^{-1}([0,1)).
\]  
\end{corollary}

With this machinery we can decompose $M$ in terms of $\Ap(M,Q)$ and $F=\langle Q\rangle$. With this decomposition, every element in $M$ can be expressed uniquely as $a+\sum_{q\in Q}\lambda_q q$, with $a\in\Ap(M,Q)$. Something similar occurs for simplicial Cohen-Macaulay affine semigroups, if we choose $Q$ to be the extremal rays of the semigroup, \cite[Theorem 1.5]{c-m}. In contrast, our monoids do not have to be finitely generated. In the finitely generated case, root-closed forces the Cohen-Macaulay property \cite{hoc}, so it is not surprising that both families of monoids share this decomposition.

\begin{theorem}\label{th:ap-decomp}
Let $M$ be a reduced root-closed inside factorial monoid with base $Q$. Then 
\[
M=\dot\bigcup_{a\in \Ap(M,Q)} a+\langle Q\rangle.
\]
\end{theorem}
\begin{proof}
Let $x\in M$. Then $n(x)x= \sum_{q\in Q} x(q) q$. If $\lambda(q,x)<1$ for all $q\in Q$, using Lemma \ref{lem:lambda-lt-1} we have $x\in\Ap(M,Q)$, and we are done.

If $\lambda(q,x)\ge 1$ for some $q$, then for all $(m,n)$ such that $m q\le_M n x$, we have that $m\ge n$. Also $q\le_M x$ by Lemma \ref{lem:lambda-ge-1}. Thus, the map $(m,n)\mapsto (m-n,n)$ is a bijection between the sets $\{(m,n)\mid m q\le_M n x\}$ and $\{(m,n)\mid m q\le_M n(x-q)\}$. This implies that $\lambda(q,x-q)=\lambda(q,x)-1$. If $\lambda(q,x-q)\ge 1$, then we repeat the argument and subtract $q$ to $x-q$. In this way we obtain a chain of strictly decreasing nonnegative real numbers, which at some point will reach an element in the interval $[0,1)$. Thus, there exists $k\in \mathbb N$ such that $x-kq\in M$ and $x-(k+1)q\not\in M$. Set $y=x-k q$, then $x=k q+y$, with $y-q\not\in M$, or equivalently $\lambda(q,y)<1$. We easily deduce that $k=\peb[\lambda(q,x)]$, as $1>\lambda(q,y)=\lambda(q,x-kq)=\lambda(x,q)-k>0$.

By \cite[Lemma 2 (a)]{inside}, $\lambda(q,x)=x(q)/n(x)$. Hence there are only finitely many $q\in Q$ such that $\lambda(q,x)\neq 0$. This lemma also states that for every $q'\in Q$, $\lambda(q',x)=\lambda(q',kq+y)=k \lambda(q',q)+\lambda(q',y)$. As a consequence of the definition of inside factorial monoid and \cite[Lemma 2]{inside}, one gets that for $q\neq q'$, $q(q')=0=q'(q)$ and thus $\lambda(q,q')=0$. It follows that $\lambda(q',y)\neq 0$ only for finitely many $q'\in Q$. By repeating the argument with $y$, we end up writing $x$ in the form $x=\sum_{q\in Q} \peb[\lambda(x,q)] q+ w$, with  $w\in \Ap(M,Q)$ (notice that $ \peb[\lambda(x,q)]$ are all zero except a finite number of them).
\end{proof}

Compare also the decomposition in Theorem \ref{th:ap-decomp} with the Stanley decomposition described in \cite[Section 4.6]{stanley}. This decomposition can also be stated in terms of the group spanned by the monoid, as the following corollary shows.

\begin{corollary}\label{cor:ap-decomp-GM}
Let $M$ be a reduced root-closed inside factorial monoid with base $Q$. Then 
\[
\mathrm{G}(M)=\dot\bigcup_{a\in \Ap(M,Q)} a+\mathrm{G}(Q).
\]
Moreover, $a+f\in M$ with $a\in \Ap(M,Q)$ and $f\in \mathrm{G}(Q)$ if and only if $f\in \langle Q\rangle$.
\end{corollary}
\begin{proof}
Let $z\in \mathrm{G}(M)$. Then $z=u-v$ for some $u,v\in M$. By Theorem \ref{th:ap-decomp}, there exist $a,b\in \Ap(M,Q)$ and $f,g\in \langle Q\rangle$ such that $u=a+f$ and $v=b+g$. Thus $z=(a-b)+(f-g)$. But $-b=(n(b)-1)b-\sum b(q)q$. Hence $z=(a+(n(b)-1)b)+(f-g-\sum b(q)q)$. Observe that $a+(n(b)-1)b\in M$, and consequently there exists $c\in\Ap(M,Q)$ and $h\in \langle Q\rangle$ such that $a+(n(b)-1)b=c+h$ (Theorem \ref{th:ap-decomp}). It follows that $z=c+(f-g+h-\sum b(q)q)\in c+\mathrm{G}(Q)$. This proves one inclusion, and the other is trivial. 

Now assume that $a\in\Ap(M,Q)$ and $f\in\mathrm{G}(Q)$ are such that $a+f\in M$. Then by Theorem \ref{th:ap-decomp}, there exists $b\in \Ap(M,Q)$ and $g\in \langle Q\rangle$ such that $a+f=b+g$. Write $f=f'-f''$ with $f',f''\in \langle Q\rangle$. Then $a+f'=b+(g+f'')$, and by Theorem \ref{th:ap-decomp} this implies that $a=b$ and $f'=g+f''$. In particular, $f=f'-f''=g\in \langle Q\rangle$.
\end{proof}


\begin{corollary}\label{cor:two-non-extreme}
Let $M$ be a reduced root-closed inside factorial monoid with base $Q$. Assume that $M$ has beside the strong atoms in $Q$ two more atoms, say $u$ and $v$. Then each $x\in M$ has a unique parametrized representation 
\[
x=\sum_{q\in Q}\lambda_q q+ m u+ n v,
\]
with $\lambda_q\in \mathbb{N}$ and parameters $m,n\in\mathbb{N}$ restricted by 
\[
m\lambda(q,u)+n\lambda(q,v)<1,
\]
for all $q\in Q$.
\end{corollary}
\begin{proof}
From Theorem \ref{th:ap-decomp}, we have that $x$ has a unique representation $x=\sum_{q\in Q}\lambda_q q+ y$, with $y\in\Ap(M,Q)$. 

First, we show that $y=m u+n v$, where $m,n\in \mathbb{N}$ are uniquely determined. That $y$ is written in this way follows from the fact that the only atoms that are not in $Q$ are $u$ and $v$. For uniqueness, suppose that $m u+n v=m' u + n' v$ for some $m',n'\in \mathbb{N}$. Assume without loss of generality that $m\ge m'$. Therefore $(m-m')u=(n'-n)v$. If $m-m'\ge n'-n$, then $((m-m')-(n'-n))u=(n'-n)(v-u)\in M$. But $M$ is root-closed, and consequently $v-u\in M$, contradicting that $u$ and $v$ are atoms. If $n'-n\ge m-m'$, we argue in the same way. 

By Lemma \ref{lem:ap-lambda}, $\lambda(q,y)<1$, and we already know that $\lambda$ is additive on the second argument. Thus $m \lambda(q,u)+ n\lambda(q,v)<1$.
\end{proof}

On $\Ap(M,Q)$ we can define the following binary operation. For $a,b\in \Ap(M,Q)$, $a+b=c+f$, with $c\in \Ap(M,Q)$ and $f\in \langle Q\rangle$ ($c$ and $f$ are unique; Theorem \ref{th:ap-decomp}). Since $f$ is unique we will denote it by $\mathrm{I}(a,b)$ (following \cite{s-cm-affine}). We set $a\oplus b=c$. 
The same construction can be performed when $M$ is a Cohen-Macaulay simplicial affine semigroup, \cite[Proposition 4]{s-cm-affine}. 

Thus, with the above notation
\[
a+b=a\oplus b + \mathrm{I}(a,b).
\]

Let $a\in \Ap(M,Q)$ and let $n$ be a nonnegative integer. We denote $\overline{na}$ as the unique element in $\Ap(M,Q)$ such that $na\in \overline{na} + \langle Q\rangle$. The existence of such element is once more guaranteed by Theorem \ref{th:ap-decomp}. Observe that
\[
na=\overline{na}+\sum_{i=1}^{n-1}\mathrm{I}(a,\overline{ia}),
\]
and that $\overline{na}=a\oplus \stackrel{n}{\cdots} \oplus a$. This notation is motivated by the fact that $\oplus$ is just $+$ modulo $\langle Q\rangle$. 

Also notice that the inverse of $a$ with respect to $\oplus$ is $\overline{(n(a)-1)a}$, since from the definition of Cale monoid $n(a)a\in \langle Q\rangle$.


Let us see some of the basic properties $\mathrm{I}$ has. They are very similar to those for the case of Cohen-Macaulay simplicial affine semigroups \cite{s-cm-affine}; and those were abstracted from the concept of Tamura's $\mathcal{N}$-semigroups \cite{tamura}.

\begin{proposition}
    Let $M$ be a reduced root-closed with Cale basis $Q$, and let $\mathrm{I}$ be defined as above.
    \begin{enumerate}
        \item For all $a,b\in \Ap(M,Q)$, $\mathrm{I}(a,b)=\mathrm{I}(b,a)$.
        \item For all $a\in \Ap(M,Q)$, $\mathrm{I}(a,0)=0$.
        \item For every $a,b,c\in \Ap(M,Q)$, $\mathrm{I}(a,b)+\mathrm{I}(a\oplus b,c)=\mathrm{I}(b,c)+\mathrm{I}(a,b\oplus c)$.
        \item If $a, b\in \Ap(M,Q)$, $a\neq 0$ and $a\oplus b=0$, then $\mathrm{I}(a,b)\not\in Q\cup\{0\}$.
        \item For all $a\in \Ap(M,Q)\setminus\{0\}$, $\sum_{i=1}^{n(a)-1}\mathrm{I}(a,\overline{i a})\neq n(a)f$ for any $f\in\langle Q\rangle$.
        \item For every positive integer $n$, every $a\in \Ap(M,Q)$ and every $f\in \mathrm{G}(Q)$, $\sum_{i=1}^{n-1}\mathrm{I}(a,\overline{i a})+n f\in \langle Q\rangle$ implies $f\in \langle Q\rangle$.
    \end{enumerate}
\end{proposition}
\newcommand{\rmI}{\mathrm{I}}
\begin{proof}
The first and second assertions follow easily from the definition. The third reflects the fact that $\oplus$ is associative. \[a+(b+c)=a+(b\oplus c+\mathrm{I}(b,c))= a\oplus(b\oplus c)+\mathrm{I}(a,b\oplus c)+\mathrm{I}(b,c)\]
and
\[(a+b)+c=(a\oplus b+\rmI(a,b)+c=(a\oplus b)\oplus c + \rmI(a\oplus b,c)+\rmI(a,b).\]
By Theorem \ref{th:ap-decomp}, $a\oplus(b\oplus c)=(a\oplus b)\oplus c$ and $\mathrm{I}(a,b\oplus c)+\mathrm{I}(b,c)=  \rmI(a\oplus b,c)+\rmI(a,b)$.

Now assume that $a\oplus b=0$ for some $a,b\in\Ap(M,Q)$. This implies that $a+b=\rmI(a,b)$. If $\rmI(a,b)=0$, then $a+b=0$, and as $M$ is reduced, we obtain $a=b=0$, a contradiction. If $\rmI(a,b)=q'\in Q$, then $q'=a+b$.  We have
$n(a)n(b)q' =n(a)n(b)(a+b) = \sum_{q\in Q} (n(a)b(q)+n(b)a(q))q$ thus $a(q)=0=b(q)$ for all $q\in Q\setminus\{ q'\}$, and $n(a)n(b)=n(a)b(q')+n(b)a(q')$.

Recall that $\lambda(q',a)=n(a)/a(q')<1$ and $\lambda(q',b)=n(b)/b(q')<1$ since $a,b\in \Ap(M,Q)$ (Lemma \ref{lem:ap-lambda}). Hence $n(a)n(b)>2n(a)n(b)$, a contradiction. This proves the fourth assertion.

Let $a\in\Ap(M,Q)$. Assume that there exists $f\in \langle Q\rangle$ such that $n(a)f=\sum_{i=1}^{n(a)-1} \rmI(a,\overline{ia})$. As $\overline{n(a)a}=0$, $n(a)a=\sum_{i=1}^{n(a)-1} \rmI(a,\overline{ia})=n(a)f$. Hence $a=f$, as $M$ is root closed. This can only be the case if $a=0=f$. The fifth assertion now follows easily.

Finally, suppose that we have $n,a,f$ under the hypothesis of the last statement. Then $a+f\in \mathrm G(M)$, and $n(a+f)=na+nf= \overline{na}+\sum_{i=1}^{n-1}\rmI(a,\overline{ia})+nf\in M$. As $M$ is root closed, $a+f\in M$, which means that $f\in \langle Q\rangle$ (Corollary \ref{cor:ap-decomp-GM}).
\end{proof}

This inspires the following structure theorem for reduced inside factorial monoids.

\begin{theorem}
    Let $G$ be a torsion Abelian group and let $F$ be the free monoid on a set $Q$. Let $I$ be a map $I:G\times G\to F$ fulfilling the following conditions:
    \begin{enumerate}
        \item for all $a,b\in G$, $I(a,b)=I(b,a)$,
        \item for all $a\in G$, $I(a,0)=0$,
        \item for every $a,b,c\in G$, $I(a,b)+I(a+ b,c)=I(b,c)+I(a,b+ c)$,
        \item if $a\in G\setminus\{0\}$, then $I(a,-a)\not\in Q\cup\{0\}$.
    \end{enumerate}    
    On $G\times F$ define the operation $(a,f)+_I(b,g)=(a+b,f+g+I(a,b))$. Then $G\times F$ with this binary operation is a reduced inside factorial monoid with basis $\{0\}\times Q$. Moreover, all elements in $\{0\}\times Q$ are irreducible, and $\Ap(G\times F,\{0\}\times Q)=G\times \{0\}$.
    
    If in addition, 
    \begin{itemize}
        \item[\textrm{(5)}] for every positive integer $n$, every $a\in G$ and every $f\in \mathrm{G}(F)$, $\sum_{i=1}^{n-1}I(a,i a)+n f\in F$ implies $f\in F$,
    \end{itemize}
    then $G\times F$ is root-closed.
\end{theorem}
\begin{proof}
The first three assertions are telling us that $G\times F$ with this binary operation is a monoid with identity element $(0,0)$. 

Assume that $(a,f)+_I(c,h)=(b,g)+_I(c,h)$. Then $a+c=b+c$ and $f+h+I(a,c)=g+h+I(b,c)$. As $G$ is a group, we get $a=b$, and thus $f+h+I(a,c)=g+h+I(a,c)$. Now, by using that $F$ is a free monoid, we deduce that $f=g$. Thus, $G\times F$ is cancellative.

Let us see now that $G\times F$ is reduced. If $(a,f)+_I(b,g)=(0,0)$, then $b=-a$ and $f+g+I(a,b)=0$. But $F$ is free, whence $f=g=I(a,-a)=0$. Property (4) of $I$ ensures that $a=0$. Hence $(a,f)=(0,0)=(b,g)$.

Now take any $(a,f)\in G\times F$. Then $\mathrm{ord}(a)(a,f)=(0,\mathrm{ord}(a)f+\sum_{i=1}^{\mathrm{ord}(a)-1} I(a,ia)) \in \langle \{0\}\times Q\rangle$. If $(0,\sum_{q\in Q}\lambda_q q)=(0,\sum_{q\in Q}\mu_q q)$ with $\lambda_q,\mu_q\in \mathbb{N}$  and all zero except finitely many of them, then as $F$ is free over $Q$, this implies that $\lambda_q=\mu_q$ for all $q\in Q$. This proves that $G\times F$ is an inside factorial monoid with basis $Q$.

Let $(a,f),(b,g)\in G\times F$ such that $(a,f)+_I(b,g)=(0,q)$ with $q\in Q$. Then $a+b=0$ and $f+g+I(a,b)=q$. By (4), $I(a,b)\neq q$, and as $F$ is free over $Q$, this implies that $I(a,b)=0$ and that either $f=0$ and $g=q$, or $f=q$ and $g=0$. In any case, by (4) again, $a=0=b$, and $(0,q)$ is an atom.

Now let $(a,0)\in G\times \{0\}$. If $(a,0)= (0,q)+_I(b,f)$ for some $(0,q)\in \{0\}\times Q$ and $(b,f)\in G\times F$, then $a=b$ and $0=q+f$, which is impossible. Thus, $G\times \{0\}\subseteq \Ap(G\times F,\{0\}\times Q)$. For the other inclusion, let $(a,f)\in \Ap(G\times F,\{0\}\times Q)$. If $f\neq0$, then $f-q\in F$ for some $q\in Q$. Hence $(a,f)=(0,q)+_I(a,f-q)\in G\times F$, contradicting that $(a,f)\in \Ap(G\times F,\{0\}\times Q)$. 

Finally assume that (5) holds. We first show that $\mathrm{G}(G\times F)=G\times \mathrm{G}(F)$ with the operation $(a,f)+_I(b,g)=(a+b,f+g+I(a,b))$. The inverse of $(a,f)$ is $(-a,-f-I(a,-a))$. If $(a,f)\in G\times\mathrm{G}(F)$, then $f=g-h$ for some $g,h\in F$, and $(a,f)=(a,g)+_I(-(0,h))$. Hence $(a,f)\in \mathrm{G}(G\times F)$. For the other inclusion, take $(a,f)\in \mathrm{G}(G\times F)$. Then $(a,f)=(b,g)+_I(-(c,h))$ for some $(b,g),(c,h)\in G\times F$, that is, $a+c=b$ and $g=f+h+I(a,c)$. Consequently $(a,f)=(a,g-h-I(a,c))\in G\times\mathrm{G}(F)$.

If for some positive integer $n$ and some $(a,f)\in G\times \mathrm{G}(F)$ we have $n(a,f)\in G\times F$, then $(na,nf+\sum_{i=1}^{n-1}I(a,ia))\in G\times F$. In particular, this implies that $nf+\sum_{i=1}^{n-1}I(a,ia)\in \mathrm{G}(F)$, and by (5), this yields $f\in F$. Thus, $(a,f)\in G\times F$, and $G\times F$ is root-closed.
\end{proof}

\section{The inner class group of a Cale monoid}\label{sec:inner-class}

Let $M$ be a reduced inside factorial monoid with Cale base $Q$. We say that $Q$ is a \emph{tame base} if for every $q\in Q$, there exists a positive integer $\ell(q)$ such that $\ell(q)\lambda(q,x)\in \mathbb{N}$ for all $x\in M$ (see \cite{cale}; we take $\ell(q)$ to be the least positive integer fulfilling this condition). We say that $M$ is a \emph{Cale monoid} if it is inside factorial with a tame base.

By \cite[Lemma 2 (b)]{inside} (and with our additive notation), we know that $\lambda(q,x+y)=\lambda(q,x)+\lambda(q,y)$ for all $x,y\in M$. For each $q\in Q$, the map $f_q(x)=\ell(q)\lambda(q,x)$ defines a monoid morphism of $M$ into $(\mathbb{N},+)$. Recall that by \cite[Lemma 2.1 (i)]{cale}, $f_q(x)=\ell(q)x(q)/n(x)$. 

Let $\varphi: M \to \mathbb{N}^{(Q)}$ defined as $\varphi(x)(q)=f_q(x)$. Then $\varphi$ is injective. If $x,y\in M$ are such that $f_q(x)=f_q(y)$ for all $q\in Q$, then $x(q)/n(x)=y(q)/n(y)$ for all $q\in Q$. Hence $x=\sum_{q\in Q} \frac{x(q)}{n(x)}q= \sum_{q\in Q} \frac{y(q)}{n(y)}q=y$.

Let $\mathrm{G}(M)$ be the quotient group of $M$. Then $\varphi$ extends to the group morphism $\varphi: \mathrm{G}(M)\to \mathbb{Z}^{(Q)}$, by defining $\varphi(a-b)=\varphi(a)-\varphi(b)$. 

For a cancellative monoid $S$ and a submonoid $T$ of $S$, we can define the following binary relation: $a\sim_T b$, $a,b\in S$ if $b-a\in \mathrm{G}(T)$. This relation is a congruence, and thus $S/\sim_T$, denoted as $S/T$, is a monoid, called the quotient of $S$ modulo $T$. It follows that $S/T$ is isomorphic to $\mathrm{G}(S)/\mathrm{G}(T)$. 

The quotient $\mathrm{Cl}(M)=\mathbb{N}^{(Q)}/\varphi(M)$ is the \emph{(outer) class group} of $M$. The \emph{inner clase group} of $M$ is defined as $\mathrm{inCl}(M)=\varphi(M)/\varphi(F)$, with $F=\langle Q\rangle$. 

Hence,
\[
\mathrm{Cl}(M)=\frac{\mathbb{Z}^{(Q)}}{\varphi(\mathrm{G}(M))}, \ \mathrm{inCl}(M)=\frac{\varphi(\mathrm{G}(M))}{\varphi(\mathrm{G}(F))}\cong \frac{M}{F}. 
\]
Thus, by the third isomorphy theorem, 
\begin{equation}
    \mathrm{Cl}(M)=\frac{\mathbb{Z}^{(Q)}/\varphi(\mathrm{G}(F))}{\mathrm{inCl}(M)}.
\end{equation}

From $\varphi(q)=\ell(q)\mathbf e_q$, for $q\in Q$ (where $\mathbf e_q$ is the function that maps everything to 0 except $q$ which is sent to 1), one obtains 
\begin{equation}
    \frac{\mathbb{Z}^{(Q)}}{\varphi(\mathrm{G}(F))}=\bigoplus_{q\in Q}\mathbb{Z}_{\ell(q)}.
\end{equation}

As a consequence of this we obtain the following corollary.

\begin{corollary}\label{cor:prod-q}
Let $M$ be a reduced root-closed Cale monoid with (finite) Cale basis $Q$. Then 
\[
\prod_{q\in Q} \ell(q) = |\mathrm{Cl}(M)|\cdot |\mathrm{inCl}(M)|.
\]
\end{corollary}

From Theorem \ref{th:ap-decomp}, we obtain another interesting consequence. 
\begin{corollary}\label{cor:inCl-Ap}
Let $M$ be a reduced root-closed Cale monoid with Cale basis $Q$. Then 
\[
    \mathrm{inCl}(M)\cong (\Ap(M,Q),\oplus)
\]
\end{corollary}

Also Theorem \ref{th:ap-decomp} and the last corollary offer a new way to see $M$. We can define on the cartesian product $\Ap(M,Q)\times F$ the following binary operation: $(a,f)+(b,g)=(a\oplus b,f+g+\mathrm{I}(a,b))$. With this operation $\Ap(M,Q)\times F$ is a monoid isomorphic to $M$.

\begin{theorem}
    Let $M$ be a reduced root-closed Cale monoid with basis $Q$. Then $M$ is isomorphic to $\Ap(M,Q)\times \langle Q\rangle$ endowed with the operation $(a,f)+(b,g)=(a\oplus b,f+g+\mathrm{I}(a,b))$.
\end{theorem}

\begin{remark}\label{rm:ell-ap}
Let $q\in Q$ and $x\in M$. Then by Theorem \ref{th:ap-decomp}, there exist unique $a\in \Ap(M,Q)$ and $f\in \\\langle Q\rangle$ such that $x=a+f$. We know that $\lambda$ is additive in the second argument, and so $\lambda(q,x)=\lambda(q,a)+\lambda(q,f)$. It also follows that $\lambda(q,f)\in \mathbb{N}$ (it is actually $f(q)$), and by Lemma \ref{lem:ap-lambda}, $\lambda(q,a)<1$. Thus if we want to compute $\ell(q)$ we only have to deal with $\lambda(q,a)$ for all $a\in \Ap(M,Q)$. As a consequence of Corollary \ref{cor:inCl-Ap}, we only have to look for least integer $k$ such $k\lambda(q,a)\in \mathbb{N}$ for all $a$ in a minimal generating set of $(\Ap(M,Q),\oplus)$. This integer will be $\ell(q)$.
\end{remark}

\section{The equation}
Let $a_1,\ldots, a_r$ be positive integers. 
Let $N=\{(x_1,\ldots, x_r)\in \NN^r \mid a_1x_1+\dots + a_{r-1}x_{r-1}=a_rx_r\}$. In light of \cite[Proposition 5.6]{cale}, $N$ is a Cale monoid.

Instead of considering $N=\{x\in \NN^r \mid a_1x_1+\cdots+a_{r-1}x_{r-1}=a_rx_r\}$, we will consider $M=\{x\in \NN^{r-1}\mid a_1x_1+\cdots+a_{r-1}x_{r-1}\equiv 0 \pmod {a_r}\}$, which as we see next is isomorphic to $N$. With this rephrasement we can apply the results in \cite{full}.

\begin{lemma}
The semigroups $N$ and ${M}$ are isomorphic.
\end{lemma}
\begin{proof}
Let $\pi:N \to M$ be projection on the first $r-1$ coordinates. Clearly $\pi$ is a monoid epimorphism. If $\pi(x)=\pi(x')$, with $x,x'\in N$, we get $a_rx_r=a_1x_1+\dots+ a_{r-1}x_{r-1} = a_1 x_1'+\dots+ a_{r-1}x_{r-1}'=a_rx_r'$. Hence $a_rx_r=a_rx_r'$ and consequently $x_r=x_r'$. Thus $\pi$ is a monoid isomorphism.
\end{proof}

The inverse of $\pi$ is $\pi^{-1}(x_1,\ldots,x_{r-1})= \left(x_1,\ldots,x_{r-1},\frac{1}{a_r}\sum_{i=1}^{r-1} x_i\right)$.

One of the side effects of this lemma is that ${M}$ is simplicial, with extremal rays $\{{q}_1,\ldots, {q}_{r-1}\}$, with $q_i=\frac{a_r}{\gcd(a_i,a_r)}\mathbf e_i$ for all $i\in\{1,\ldots, r-1\}$ (in this context, $\mathbf e_i$ denotes the $i$th row of the $(r-1)\times(r-1)$ identity matrix).

Now \cite{hoc} can be used in combination with \cite{c-m}. It also follows easily that $M$ is a Cale monoid with basis $Q=\{q_1,\ldots, q_{r-1}\}$. This has also another advantage. We may consider now the $a_i$ modulo $a_r$.

The monoid $M$ is a full affine semigroup, and it is generated by the set of minimal nonzero elements in $M$ with respect to the usual partial ordering in $\mathbb{N}^{r-1}$. This set is known in the literature as a \emph{Hilbert basis} of $M$.

The following result implicitly appears in \cite[Theorem 10]{full}.

\begin{lemma}\label{carac-ap-inter}
Under the standing hypothesis,
\[
\Ap(M,Q)=M\cap \prod\nolimits_{i=1}^{r-1}[0,a_r/\gcd(a_i,a_r)).
\]
\end{lemma}
\begin{proof}
Let $x\in \Ap(M,Q)$, and assume that for some $i$, $x_i\ge a_r/\gcd(a_i,a_r)$. Then $x-q_i\in \NN^{r-1}$, but this means that $x-q_i\in M$, contradicting that $x\in \Ap(M, q_i)$.

The other inclusion is clear, since for every $x\in M$ with $x_i<a_r/\gcd(a_i,a_r)$ for all $i$ one cannot have $x-q_i\in M$.
\end{proof}

As a consequence of the above lemma and Theorem \ref{th:ap-decomp}, we obtain the following consequence.

\begin{corollary}\label{cor:ap-decomp-single-eq}
Let $a_1,\ldots, a_r$ be positive integers, with $r>2$. Let $M$ be the set of nonnegative integer solutions of $a_1x_1+\cdots+ a_{r-1}x_{r-1}\equiv 0\pmod {a_r}$. Set $F=\langle \frac{a_r}{\gcd(a_1,a_r)}\mathbf e_1,\ldots, \frac{a_r}{\gcd(a_{r-1},a_r)}\mathbf e_{r-1}\rangle$, and let $A=\Ap(M,Q)$. Then 
\[
M=\bigcup_{a\in A} a+F,
\]
and this union is disjoint.
\end{corollary}


\begin{corollary}\label{cor:hb-r+1}
Let $a_1,\ldots, a_r$ be positive integers, with $r>2$. Let $M$ be the set of nonnegative integer solutions of $a_1x_1+\cdots+ a_{r-1}x_{r-1}\equiv 0\pmod {a_r}$. Assume that the Hilbert basis of $M$ is of the form $\{\frac{a_r}{\gcd(a_1,a_r)}\mathbf e_1,\ldots, \frac{a_r}{\gcd(a_{r-1},a_r)}\mathbf e_{r-1}, u, v\}$, for some $u,v$. Then each $x\in M$ has a unique parametrized representation 
\[
x=\sum_{i=1}^{r-1}\lambda_i \frac{a_r}{\gcd(a_i,a_r)}\mathbf e_i + mu+ nv,
\]
with $\lambda_i\in \mathbb N$ and parameters $m,n\in \mathbb{N}$ restricted to 
\[
m u_i +n v_i< \frac{a_r}{\gcd(a_i,a_r)},
\]
for all $i\in\{1,\ldots,r-1\}$.
\end{corollary}
\begin{proof}
Apply Corollary \ref{cor:two-non-extreme} and the particular shape of $A=\Ap(M,Q)$ in Corollary \ref{cor:ap-decomp-single-eq}.
\end{proof}

\begin{remark}
Notice that both Corollaries \ref{cor:two-non-extreme} and \ref{cor:hb-r+1} hold if the Hilbert basis consists in the extreme rays plus an extra element, say $u$. The decomposition and restrictions of the parameters will be the same suppressing the terms in $v$.
\end{remark}

\begin{example}\label{ex-1}
Let $N$ be the set of nonnegative integer solutions of $4x+5y=7z$. Then we know that $M$ is the set of nonnegative integer solutions of $4x+5y\equiv 0 \pmod 7$. The extremal rays of $M$ are $(7,0)$ and $(0,7)$, that is, the Cale base of $M$ is $Q=\{(7,0),(0,7)\}$. Since $M$ is also a full affine semigroup, $\Ap(M,Q)= M\cap [0,6]^2$. The elements in this set are 
\[
\{(0, 0 ), ( 1, 2 ), ( 2, 4 ), ( 3, 6 ), ( 4, 1), ( 5, 3 ), ( 6, 5 )\}.
\]
In order to obtain the Hilbert basis of $M$ it suffices to take $(7,0)$, $(0,7)$ and the minimal nonzero elements of $M\cap [0,6]^2$: $(1,2)$ and $(4,1)$. Thus $H=\{(7,0),(0,7),(1,2),(4,1)\}$ is a Hilbert basis of $M$.


By Corollary \ref{cor:hb-r+1}, with $u=(1,2)$ and $v=(4,1)$, the unique representation of any element $x\in M$ is 
\[
x=\lambda_1(7,0)+\lambda_2(0,7)+m(1,2)+n(4,1),
\]
with $\lambda_1,\lambda_2\in\mathbb N$ and parameters $m,n\in \mathbb{N}$ restricted by 
\[
m+4n<7\hbox{ and } 2m+n<7.
\]
The latter restrictions are equivalent to $(m,n)\in \{0,1,2\}\times \{0,1\} \cup \{(3,0)\}$, and the integers $\lambda_1$,  $\lambda_2$, $n$, $m$ are unique (compare with \cite{elliot}, page 368). If we do not put restrictions on $n$ and $m$, then we loose uniqueness.




Set $q_1=(7,0)$ and $q_2=(0,7)$. As $7(1,2)=(7,0)+2(0,7)$, we obtain $\lambda(q_1,(1,2))=\frac{1}{7}$ and $\lambda(q_2,(1,2))=\frac{2}7$. In a similar way we can compute $\lambda(q_i,a)$ for $i\in\{1,2\}$ and $a\in \Ap(M,Q)$. It follows that $\ell(q_1)=\ell(q_2)=7$. By Corollary \ref{cor:inCl-Ap}, $\mathrm{inCl}(M)$ is isomorphic to $\Ap(M,Q)$ endowed with the operation $\oplus$. This group is isomorphic to $\mathbb{Z}_7$. 

In our setting $\mathbb{Z}^{(Q)}=\mathbb{Z}^2$ and $\varphi:M\to \mathbb{Z}^2$, $\varphi(x)=(7\lambda(q_1,x),7\lambda(q_2,x))$. Also $\varphi(\mathrm{G}(M))=\mathrm{G}(\varphi(M))=\mathrm{G}(\varphi(H))$, so we can describe $\varphi(\mathrm{G}(M))$ by computing the images of the elements in $H$ via $\varphi$ and then taking the group spanned by them. A quick calculation shows that $\varphi((7,0))=(7,0)$, $\varphi((0,7))=(0,7)$, $\varphi((1,2))=(1,2)$, $\varphi((4,1))=(4,1)$, and consequently $\varphi$ is just the inclusion of $M$ inside $\mathbb{N}^2$. The group spanned by $\{(7,0),(0,7),(1,2),(4,1)\}$ equals the group $\mathrm{G}((1,2),(3,-1))$ with invariant factors $1$, $7$. Thus $\mathrm{Cl}(M)\cong \mathbb{Z}_7$.
\end{example}

\section{The two dimensional case}

In this section we focus in the case $ax+by=cz$, or equivalently, $ax+by\equiv 0\pmod c$. We may assume that $\gcd(a,b,c)=1$ (if not, we divide the equation by this amount).

\begin{lemma}\label{ap-two}
Let $c$ be a positive integer and let $a,b\in \{1,\ldots,c-1\}$ with $\gcd(a,b,c)=1$. Let $M$ be the set of nonnegative integer solutions of $a x+ b y\equiv 0  \pmod c$. Let $Q=\{(\frac{c}{\gcd(a,c)},0),(0,\frac{c}{\gcd(b,c)})\}$.Then 
\[
\Ap(M,Q)=\left\{ \left(\gcd(b,c)i, - i d a  \mod\left(\frac{c}{\gcd(b,c)}\right)\right) ~\middle|~ i\in \left\{0,\ldots, \frac{c}{\gcd(a,c)\gcd(b,c)}-1\right\}\right\},
\]
where $d=(b/\gcd(b,c))^{-1}\bmod (c/\gcd(b,c))$.
\end{lemma}
\begin{proof}
We already know by Lemma \ref{carac-ap-inter} that \[\Ap(M,\{(c/\gcd(a,c),0),(0,c/\gcd(b,c)\})=M \cap [0,c/\gcd(a,c))\times [0,c/\gcd(b,c)).\] 

We are going to prove that 
\[
\Ap(M,\{(c/\gcd(a,c),0),(0,c/\gcd(b,c)\})=\left\{ (i,j) ~\middle|~ \begin{matrix}i\in \{0,\ldots, c/\gcd(a,c)-1\}\cap \gcd(b,c)\mathbb{N},\\  j= - d a (i/\gcd(b,c)) \bmod (c/\gcd(b,c))\end{matrix}\right\},
\]
and then the rest of the proof follows by a change of indices.

Observe that if $(i,j),(k,l)\in M \cap [0,c/\gcd(a,c))\times [0,c/\gcd(b,c))$, with $(i,j)\neq (k,l)$, then $i\neq k$ and $j\neq l$. Assume to the contrary and without loss of generality that $i=k$ and $j>l$. Then $(0,j-l)\in M$, or equivalently $(j-l)b\equiv 0\pmod c$. But then $j-l$ is a multiple of $c/\gcd(b,c)$ in $[0,c/\gcd(b,c))$, which implies $j-l=0$, a contradiction.  

Let us see for which $i\in \{0,\ldots, c/\gcd(a,c)-1\}$ there is $j\in\{0,\ldots, c/\gcd(b,c)-1\}$ such that $a i+b j\equiv 0\pmod c$. The equation in $j$, $bj\equiv -ai \pmod c$, has a solution if and only if $\gcd(b,c)$ divides $-ai$, but $\gcd(\gcd(b,c),a)=\gcd(a,b,c)=1$, and so this occurs if and only if $\gcd(b,c)\mid i$.

Thus, for each $i\in \{0,\ldots, c/\gcd(a,c)-1\}\cap \gcd(b,c)\mathbb{N}$ there exists a unique $j\in \{0,\ldots, c/\gcd(b,c)-1\}$ such that $(i,j)\in M$. 

If $(i,j)\in M$, we already know that $\gcd(b,c)\mid i$. Also there must exist $r\in\mathbb N$ such that $ai+bj=rc$. Dividing by $\gcd(b,c)$, we have $ai/\gcd(b,c)+b/\gcd(b,c)j=rc/\gcd(b,c)$, or equivalently $ai/\gcd(b,c)+b/\gcd(b,c)j\equiv 0 \pmod {c/\gcd(b,c)}$. From here, we obtain that $j$ can be expressed as $j=- (b/\gcd(b,c))^{-1} a i/\gcd(b,c) \bmod (c/\gcd(b,c))$. 
\end{proof}


\begin{theorem}\label{th:two-cl-cyclic}
    Let $c$ be a positive integer and let $a,b\in \{1,\ldots,c-1\}$ with $\gcd(a,b,c)=1$. Let $M$ be the set of nonnegative integer solutions of $a x+ b y\equiv 0  \pmod c$. Then
    \[
    \mathrm{inCl}(M)\cong \mathrm{Cl}(M) \cong \mathbb{Z}_{\frac{c}{\gcd(a,c)\gcd(b,c)}}.
    \]
\end{theorem}
\begin{proof}
The equality $\mathrm{Cl}(M)\cong \mathbb{Z}_{\frac{c}{\gcd(a,c)\gcd(b,c)}}$ is a consequence of \cite[Theorem 1.3]{d-half}. By Corollary \ref{cor:inCl-Ap} and Lemma \ref{ap-two}, we get that $\mathrm{inCl}(M)$ is a cyclic group of size $c/(\gcd(a,c)\gcd(b,c))$ and thus it is isomorphic to $\mathrm{Cl}(M)$.
\end{proof}

\begin{remark}
For $a,b,c$ positive integers with $\gcd(a,b,c)=1$ we have that the set of solutions of $ax+by\equiv 0\bmod c$ is a Cale monoid with tame basis $\{q_1=(c/\gcd(a,c),0), q_2=(0,c/\gcd(b,c))\}$. Also by Lemma \ref{ap-two}, the group $\Ap(M,Q)$ (with the operation $\oplus$) is generated by an element of the form $(\gcd(b,c),k)$ with $k$ a positive integer less than $c/\gcd(b,c)$. By Remark \ref{rm:ell-ap}, in order to compute $\ell(q_1)$ and $\ell(q_2)$ we only have to look at the denominators of the rational number $\lambda(q_1,(\gcd(b,c),k))$ and $\lambda(q_2,(1,k))$. From $\frac{c}{\gcd(b,c)}(1,k)=\gcd(a,c)q_1+ k q_2$, $\lambda(q_1,(1,k))=\gcd(a,c)\gcd(b,c)/c=1/(c/\gcd(a,c)\gcd(b,c))$, and thus $\ell(q_1)=c/(\gcd(a,c)\gcd(b,c))$. But now Corollary \ref{cor:prod-q} states that $\ell(q_1)\ell(q_2)=|\mathrm{Cl}(M)||\mathrm{inCl}(M)|$ and we know by Theorem \ref{th:two-cl-cyclic} that the equalities $|\mathrm{Cl}(M)|=|\mathrm{inCl}(M)|=c/(\gcd(a,c)\gcd(b,c))$ hold. Therefore $\ell(q_2)=c/(\gcd(a,c)\gcd(b,c))$. 
\end{remark}

\begin{example}\label{mainexample}
Let us go back to Example \ref{ex-1}, now applying Lemma \ref{ap-two}. We have $a=4$, $b=5$ and $c=7$. Then $b^{-1}\bmod c= 3$. Thus by Lemma \ref{ap-two}, the elements in   $\Ap(M,\{(c,0),(0,c)\})$ are of the form \[(i,-3\times 4\times i \pmod 7)\] for $i\in \{0,\ldots, c-1\}$. For $i=0$, we get $(0,0)$, for $i=1$, we obtain $(1, -12 \bmod 7)=(1,2)$, and so on.

Then to obtain a Hilbert basis we filtered the set obtained, looking for the second coordinate less than the last second coordinate obtained, i.e. as the set obtained using the algorithm is \[\{(0, 7 ), ( 1, 2 ), ( 2, 4 ), ( 3, 6 ), ( 4, 1), ( 5, 3 ), ( 6, 5 ), (7,0)\}.\] We construct the Hilbert basis starting with $\{(0,7),(1,2)\}$, the point $(2,4)$ has second coordinate bigger than 2 (the second coordinate of the last point added, then we do not need it, as well $(3,6)$). However $(4,1)$ has the second coordinate less than 2, so we add it. Finally $(5,3)$ and $(6,5)$ have second coordinate bigger than 1 so we ignore them. Adding $(7,0)$ we construct the Hilbert basis $H=\{(0,7),(1,2),(4,1),(7,0)\}$.  

Note that multiplying by $2\equiv 4^{-1}\bmod 7$ both $a$ and $b$ we obtain the same Ap\'ery set and the same conclusions, that is, the equation $x+3 y\equiv 0\bmod 7$ is equivalent to the above.
\end{example}

Observe that Corollary \ref{cor:ap-decomp-single-eq} and Lemma \ref{ap-two} provide a way to express any solution of $ax+by=c z$ (equivalently $ax+by\equiv 0 \pmod c$) in a unique way as $a+\lambda q_1+\mu q_2$, with $a$ in the Ap\'ery set of $\{q_1,q_2\}$, and $\lambda, \mu \in\mathbb N$. However this is not the approach given by Elliott; he was looking for a (unique) description of the solutions of the form $\gamma_1 a_1+ \cdots + \gamma_t a_t + \lambda q_1 +\mu q_2$, with $\lambda,\mu \in \mathbb N$ and $\gamma_i \subset J_i$, with $J_i$ an interval of nonnegative integers. We gave such a description in Example \ref{ex-1}. Next we present a family of equations where we can provide a ``unique'' description of the solutions in Elliott's sense.

\begin{example}
Consider the equation $a x+ y\equiv 0 \pmod{2 a+1}$ with $a\in\mathbb Z^+$. We have $\gcd(a,2 a+1)=1$. It is no difficult to see that the Hilbert basis is $H=\{(0,2 a+1),(1,a+1),(2,1),(2 a+1,0)\}$, and also
\begin{align*}
\Ap(M,\{(2 a+1,0),(0,2 a+1)\}) = &\{(2i,i) \mid i\in \{0,\ldots,a-1\}\} \\ & \cup\{ (2i+1,a+i+1) \mid i\in\{0,\ldots,a-1\}\}.
\end{align*}


We apply Corollary \ref{cor:hb-r+1} with $u=(1,a+1)$ and $v=(2,1)$. It follows that any solution of $ax+y \equiv 0 \pmod{2 a+1}$ is expressed (uniquely) as $\lambda_1(0,2a+1)+\lambda_2(2a+1,0)+m (1,a+1)+n (2,1)$, with $\lambda_i\in \mathbb{N}$ and $m,n\in \mathbb{N}$ parameters restricted to $m+2a<2a+1$ and $m(a+1)+n<2a+1$. These restrictions are equivalent to $(m,n)\in \{0\}\times[0,a]\cup \{1\}\times[0,a-1]$.
\end{example}
Recall that Example \ref{mainexample} corresponds to the equation $x+ b y\equiv 0\pmod{2 b+1}$ by taking $b=3$; this case is dual to the above example and can be worked out in a similar way.

Another useful algorithm to find the set of minimal solution, or Hilbert basis, for this type of equations can be found in \cite[Figure 3]{FT}, which computes all minimal solutions by means of geometrical tools. The authors called it \textit{slopes algorithm} because it calculates the slopes of the lines containing the minimal solutions.

\end{document}